\documentclass[11pt, reqno]{amsart}

\usepackage{amsmath,amssymb,amsthm,mathtools}
\usepackage[
  left=1.10in,
  right=1.10in,
  top=1.2in,
  bottom=1.1in
]{geometry}
\usepackage{microtype}
\usepackage{placeins}
\usepackage{needspace}
\usepackage[hidelinks]{hyperref}
\usepackage{tikz}
\usetikzlibrary{arrows.meta,decorations.pathreplacing}

\tikzset{
  strand/.style   = {line width=.7pt,line cap=round,line join=round},
  over/.style     = {preaction={draw,white,line width=2.6pt}},
  bracelab/.style = {font=\scriptsize},
  panellab/.style = {font=\normalsize}
}

\newcommand{\Tbox}[2]{%
  \draw[strand,fill=white] ({#1-.55},-.45) rectangle ({#1+.55},.45);
  \node[font=\small] at ({#1},0) {$#2$};}
\newcommand{\Tarc}[4][6pt]{\draw[strand,rounded corners=#1]
  ({#2},.45) -- ({#2},{.45+#4}) -- ({#3},{.45+#4}) -- ({#3},.45);}
\newcommand{\Barc}[4][6pt]{\draw[strand,rounded corners=#1]
  ({#2},-.45) -- ({#2},{-.45-#4}) -- ({#3},{-.45-#4}) -- ({#3},-.45);}
\newcommand{\Dslot}[1]{%
  \draw[strand] ({#1-.45},.45)--({#1-.45},.16);
  \draw[strand] ({#1+.45},.45)--({#1+.45},.16);
  \draw[strand] ({#1-.45},-.45)--({#1-.45},-.16);
  \draw[strand] ({#1+.45},-.45)--({#1+.45},-.16);
  \node[font=\small] at ({#1},0) {$\cdots$};}
\newcommand{\Ubrace}[3]{\draw[line width=.6pt,decorate,
  decoration={brace,amplitude=5pt,mirror}] ({#1},{#3}) -- ({#2},{#3});}

\newtheorem{theorem}{Theorem}[section]
\newtheorem{corollary}[theorem]{Corollary}
\newtheorem{proposition}[theorem]{Proposition}
\newtheorem{lemma}[theorem]{Lemma}
\theoremstyle{remark}
\newtheorem{remark}[theorem]{Remark}

\newcommand{\Z}{\mathbb Z}
\newcommand{\Q}{\mathbb Q}

\title[Exceptional pretzel knots are not algebraically slice]
{Exceptional pretzel knots are not algebraically slice}
\author{Suman Saurabh}
\email{realsumansaurabh@gmail.com}
\hypersetup{
  pdftitle={Exceptional pretzel knots are not algebraically slice},
  pdfauthor={Suman Saurabh},
  pdfkeywords={pretzel knots, Alexander polynomial, Fox--Milnor factorization, knot concordance}
}

\begin{document}

\begin{abstract}
We consider the three-strand pretzel knot
$P_a=P(a,-a-2,-(a+1)^2/2)$ and prove that, when
$a\equiv1\pmod4$, its Alexander polynomial fails the Fox--Milnor
condition.  The same remains true after adding any number of opposite
odd parameter pairs $q,-q$ and reordering the parameters.  Combined
with work of Kim--Lee--Song, this gives the same conclusion for every
odd $a\ge3$.  In particular, no knot in Lecuona's exceptional family
is algebraically or topologically slice.  Together with earlier work
of Lecuona, Miller, and Kim--Lee--Song, this completes the
slice--ribbon and topological-sliceness classifications for
three-strand pretzel knots and removes the nonexceptional hypothesis
from the Lecuona--Wand classification of prime fibered ribbon pretzel
knots.
\end{abstract}

\maketitle

\section{Introduction}

For nonzero integers $p_1,\ldots,p_n$, let $P(p_1,\ldots,p_n)$ denote
the standard pretzel link with $p_i$ signed half-twists in the $i$th
twist region.  We use the sign convention of \cite{Lecuona2015} and
consider only parameter lists that give knots.

A knot is \emph{smoothly slice} if it bounds a smoothly embedded disk
in the four-ball, and it is \emph{ribbon} if such a disk may be chosen
with no interior local maxima for the radial function.  It is
\emph{topologically slice} if it bounds a locally flat disk in the
four-ball, and it is \emph{algebraically slice} if its Seifert form is
metabolic.  These notions satisfy
\[
\text{ribbon}\Longrightarrow\text{smoothly slice}
\Longrightarrow\text{topologically slice}
\Longrightarrow\text{algebraically slice}.
\]
The slice--ribbon conjecture asserts that the first implication is an
equivalence.

For Laurent polynomials $p,q\in\Z[t,t^{-1}]$, write $p\doteq q$ if
$p=\pm t^nq$ for some $n\in\Z$.  An algebraically slice knot $K$
satisfies the Fox--Milnor condition
\begin{equation}\label{eq:FMintro}
\Delta_K(t)\doteq f(t)f(t^{-1})
\end{equation}
for some $f\in\Z[t,t^{-1}]$; see Fox--Milnor \cite{FoxMilnor1966} and
Levine \cite{Levine1969}.  Thus failure of \eqref{eq:FMintro} rules out
both algebraic and topological sliceness.

Greene and Jabuka \cite{GreeneJabuka2011} proved the slice--ribbon
conjecture for three-strand pretzel knots with all parameters odd.
For knots with one even parameter, Lecuona \cite{Lecuona2015} reduced
the remaining cases to
\begin{equation}\label{eq:coreintro}
P_a=P\!\left(a,-a-2,-\frac{(a+1)^2}{2}\right),
\qquad a\ge3\ \text{odd},
\end{equation}
and their mirrors.  Miller \cite{Miller2017} subsequently sharpened
the topological-sliceness classification.  Kim, Lee, and Song
\cite[Theorems~3.3 and~4.2]{KLS2022} showed that
$\Delta_{P_a}(t)$ fails the Fox--Milnor condition when
$a\equiv3\pmod4$ or $a\equiv5\pmod8$.  After these results, the
remaining core parameters in the three-strand classification were
\[
a\equiv1,97\pmod{120}.
\]

Both residual classes lie in $a\equiv1\pmod4$.  We prove a uniform
Fox--Milnor obstruction for this whole congruence class.  In fact, the
result holds for a larger family obtained by adjoining arbitrary
opposite odd parameter pairs and reordering the parameters.  For
positive odd integers $q_1,\ldots,q_r$, let
$\mathcal P(a;q_1,\ldots,q_r)$ denote the collection of pretzel knots
obtained by arbitrarily reordering
\begin{equation}\label{eq:pairedfamily}
a,\ -a-2,\ -\frac{(a+1)^2}{2},\
q_1,-q_1,\ldots,q_r,-q_r.
\end{equation}
We allow $r=0$, so that $P_a$ is included.

\begin{theorem}\label{thm:main}
Let $a\ge3$ be odd with $a\equiv1\pmod4$, and let
$q_1,\ldots,q_r$ be positive odd integers.  For every
$K\in\mathcal P(a;q_1,\ldots,q_r)$, the Alexander polynomial
$\Delta_K(t)$ does not admit a Fox--Milnor factorization.
Consequently $K$ is not algebraically slice and, in particular, is
not topologically slice.
\end{theorem}

Theorem~\ref{thm:main} is stronger than what is needed for the two
remaining three-strand classes.  Combining it with
\cite[Theorem~4.2]{KLS2022} and
Proposition~\ref{prop:pairedreduction} gives the corresponding
statement for every odd $a$.

\begin{corollary}\label{cor:allpaired}
Let $a\ge3$ be odd, and let $q_1,\ldots,q_r$ be positive odd integers.
Every knot in $\mathcal P(a;q_1,\ldots,q_r)$ fails the Fox--Milnor
condition and hence is neither algebraically nor topologically slice.
\end{corollary}

Lecuona's exceptional set $\mathcal E$ is a subfamily of
\eqref{eq:pairedfamily}, with $q_i\ge3$ odd and
\[
a\equiv1,11,37,47,49,59\pmod{60};
\]
see \cite[Theorem~1.1 and Conjecture~1.3]{Lecuona2015}.
Corollary~\ref{cor:allpaired} therefore eliminates the entire
exceptional set.

\begin{corollary}\label{cor:exceptional}
No knot in Lecuona's exceptional set $\mathcal E$ is algebraically
slice.  Consequently, Conjecture~1.3 of \cite{Lecuona2015} holds and
the nonexceptional hypothesis may be removed from
\cite[Theorem~1.1]{Lecuona2015}.  No member of the residual exceptional
family of \cite{LecuonaWand2026} is ribbon.
\end{corollary}

Together with the earlier classifications of Lecuona, Miller, and
Kim--Lee--Song, this closes the remaining three-strand cases.

\begin{corollary}[Slice--ribbon for three-strand pretzel knots]
\label{cor:3strand}
A three-strand pretzel knot is smoothly slice if and only if it is
ribbon.  Moreover, it is topologically slice if and only if it is
ribbon or has trivial Alexander polynomial.
\end{corollary}

Corollary~\ref{cor:exceptional} also removes the exceptional proviso
from the recent classification of prime fibered ribbon pretzel knots.

\begin{corollary}\label{cor:LW}
The classification of prime fibered ribbon pretzel knots up to
reordering of parameters in \cite[Theorem~1.1]{LecuonaWand2026}
holds without the nonexceptional hypothesis.
\end{corollary}

We finish the introduction with an outline of the proof.
Section~\ref{sec:reduction} shows that adjoining opposite parameter
pairs does not change whether the Alexander polynomial admits a
Fox--Milnor factorization, reducing the problem to the knots $P_a$.
A Cayley transform and a rational parametrization then produce an even
integral polynomial $Q$ with explicit endpoint values and with
$Q(s)>0$ for $s\in[a,a+2]$.  A hypothetical Fox--Milnor
factorization descends to
\[
Q(s)=-H(s)H(-s)
\]
for a monic polynomial $H\in\Z[s]$.  An elementary content argument
determines the absolute values of $H$ at the endpoints, and a
congruence modulo $(a+1)/2$ forces the endpoint signs to be opposite.
Thus $H$ has a zero in $(a,a+2)$, contradicting the positivity of
$Q$ there. \newline

\textbf{Acknowledgement:} The author is grateful to Professor Peter Teichner for his notes on \emph{Slice knots},
which provided an introduction to the subject and were helpful
when returning to the background for this project.
\newline

\textbf{AI Declaration:}
The mathematical content presented in this manuscript was produced entirely without the aid of AI; AI tools were utilized strictly for minor language editing and to help write TikZ code for the figures.

\section{Paired extensions and the Fox--Milnor condition}\label{sec:reduction}

If $g\in\Z[t,t^{-1}]$ is irreducible, call it \emph{self-reciprocal} if $g(t)\doteq g(t^{-1})$.  A reciprocal Laurent polynomial admits a Fox--Milnor factorization precisely when every irreducible self-reciprocal factor occurs with even multiplicity; see, for example, \cite[Remark~2.2]{KLS2022}.

\begin{lemma}\label{lem:concordanceFM}
If two knots $K$ and $J$ are concordant, then $\Delta_K(t)$ admits a Fox--Milnor factorization if and only if $\Delta_J(t)$ does.
\end{lemma}

\begin{proof}
Write $-J$ for the reversed mirror of $J$.  Since $K\mathbin{\#}(-J)$ is slice,
\[
\Delta_K(t)\Delta_J(t)\doteq\Delta_{K\mathbin{\#}(-J)}(t)
\]
admits a Fox--Milnor factorization.  Thus, for each self-reciprocal irreducible Laurent polynomial, the sum of its multiplicities in $\Delta_K$ and $\Delta_J$ is even.  Those two multiplicities consequently have the same parity.  The factor-multiplicity criterion above proves the equivalence.
\end{proof}

For an adjacent pair $(q,-q)$, with $q\ge3$ odd, we use the
standard ribbon reduction of \cite[Proposition~2.1]{Lecuona2015}.
Figure~\ref{fig:global-ribbon-move} records its global effect, and
Figure~\ref{fig:local-ribbon-move} shows the local band surgery.
It transforms $P(q,-q,R)$ into the split
union
\[
U\sqcup P(R),\]
where $U$ is an unknot and $P(R)$ is obtained by deleting the pair
$q,-q$.  Capping off $U$ yields a concordance from $P(q,-q,R)$ to
$P(R)$.

%%=====================================================================
%%  FIGURE 1 : the global move
%%=====================================================================
\begin{figure}[!b]
\centering
\begin{tikzpicture}[scale=.71]

%%---------------- (a)  P(q,-q,R) ----------------
\node[panellab] at (4,3.00) {(a)};

\Tbox{0}{q}\Tbox{1.6}{-q}\Tbox{3.2}{r_1}\Tbox{4.8}{r_2}\Tbox{8.0}{r_m}
\Dslot{6.4}

\Tarc{0.30}{1.30}{.75}\Barc{0.30}{1.30}{.75}
\Tarc{1.90}{2.90}{.75}\Barc{1.90}{2.90}{.75}
\Tarc{3.50}{4.50}{.75}\Barc{3.50}{4.50}{.75}
\Tarc{5.10}{5.95}{.70}\Barc{5.10}{5.95}{.70}
\Tarc{6.85}{7.70}{.70}\Barc{6.85}{7.70}{.70}
\Tarc[12pt]{-0.30}{8.30}{1.85}\Barc[12pt]{-0.30}{8.30}{1.85}

\draw[line width=.6pt,dashed,rounded corners=4pt]
      (-0.95,-1.35) rectangle (2.40,1.35);

\Ubrace{-0.55}{2.15}{-2.50}
\node[bracelab] at (0.80,-3.12) {adjacent opposite pair};
\Ubrace{2.65}{8.55}{-2.50}
\node[bracelab] at (5.60,-3.12) {$R=(r_1,\dots,r_m)$};

%%---------------- arrow ----------------
\draw[line width=.7pt,-{Stealth[length=2.2mm,width=1.6mm]}] (9.35,0) -- (11.85,0);
\node[font=\small,align=center] at (10.60,0.95) {ribbon move\\and isotopy};

%%---------------- (b)  U  and  P(R) ----------------
\node[panellab] at (17.75,3.00) {(b)};

\draw[strand] (13.30,0.35) circle[radius=.68];
\node[font=\small] at (13.30,-0.75) {$U$};

\Tbox{15.2}{r_1}\Tbox{16.8}{r_2}\Tbox{20.0}{r_m}
\Dslot{18.4}
\Tarc{15.50}{16.50}{.75}\Barc{15.50}{16.50}{.75}
\Tarc{17.10}{17.95}{.70}\Barc{17.10}{17.95}{.70}
\Tarc{18.85}{19.70}{.70}\Barc{18.85}{19.70}{.70}
\Tarc[12pt]{14.90}{20.30}{1.85}\Barc[12pt]{14.90}{20.30}{1.85}

\Ubrace{14.65}{20.55}{-2.50}
\node[bracelab] at (17.60,-3.12) {$P(R)$};

\end{tikzpicture}
\caption{
A ribbon move on the adjacent opposite pair $(q,-q)$, with $q\ge3$ odd, followed by
isotopy, removes that pair and transforms $P(q,-q,R)$ into the split
union $U\sqcup P(R)$, where $U$ is an unknot.
}
\label{fig:global-ribbon-move}
\end{figure}
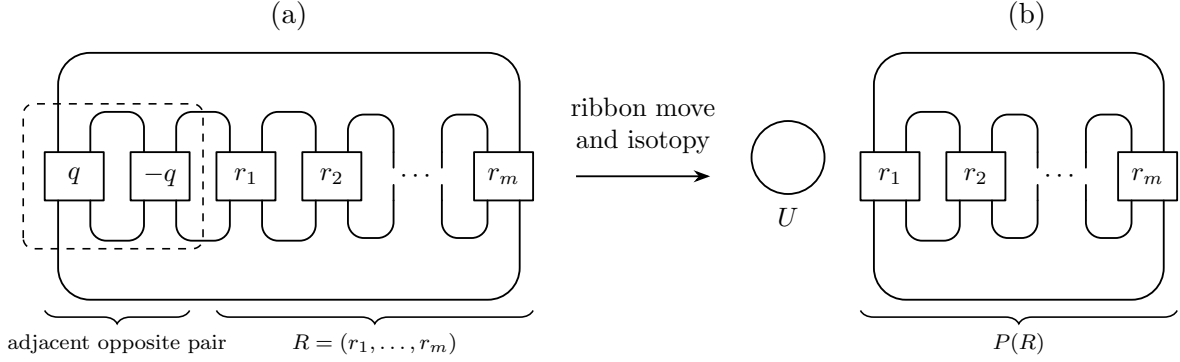

% Figure 2: local band surgery and cancellation.
\begin{figure}[!t]
\centering
\begin{tikzpicture}[scale=.82]
% Crossing conventions: the NW--SE arc is over in a positive box.
\newcommand{\pcross}[6]{%
  \draw[strand,#6] (#2,#3) .. controls (#2,{#3-.18}) and (#1,{#4+.18}) .. (#1,#4);
  \draw[strand,over,#5] (#1,#3) .. controls (#1,{#3-.18}) and (#2,{#4+.18}) .. (#2,#4);}
\newcommand{\ncross}[6]{%
  \draw[strand,#5] (#1,#3) .. controls (#1,{#3-.18}) and (#2,{#4+.18}) .. (#2,#4);
  \draw[strand,over,#6] (#2,#3) .. controls (#2,{#3-.18}) and (#1,{#4+.18}) .. (#1,#4);}
\newcommand{\pairbase}[1]{%
  \pcross{-1}{-.4}{1.5}{.95}{black}{black}
  \ncross{.4}{1}{1.5}{.95}{black}{black}
  \draw[strand] (-1,.95)--(-1,.45) (1,.95)--(1,.45);
  \pcross{-1}{-.4}{.45}{-.15}{black}{#1}
  \ncross{.4}{1}{.45}{-.15}{#1}{black}
  \draw[strand,#1] (-1,-.15)--(-1,-.4) (1,-.15)--(1,-.4);
  \draw[strand] (-.4,-.15)--(-.4,-.4) (.4,-.15)--(.4,-.4);
  \pcross{-1}{-.4}{-.4}{-1}{#1}{black}
  \ncross{.4}{1}{-.4}{-1}{black}{#1}
  \draw[strand] (-.4,1.5) .. controls (-.4,2.05) and (.4,2.05) .. (.4,1.5);
  \draw[strand,#1] (-.4,-1) .. controls (-.4,-1.65) and (.4,-1.65) .. (.4,-1);
  \draw[strand] (-1,1.5) .. controls (-1,1.95) and (-1.65,1.9) .. (-1.65,2.2);
  \draw[strand] (1,1.5) .. controls (1,1.95) and (1.65,1.9) .. (1.65,2.2);
  \draw[strand] (-1,-1) .. controls (-1,-1.7) and (-1.65,-1.9) .. (-1.65,-2.2);
  \draw[strand] (1,-1) .. controls (1,-1.7) and (1.65,-1.9) .. (1.65,-2.2);}
\newcommand{\localends}{%
  \draw[strand,densely dotted] (-1.65,2.2)--(-1.65,2.58) (1.65,2.2)--(1.65,2.58)
  (-1.65,-2.2)--(-1.65,-2.58) (1.65,-2.2)--(1.65,-2.58);}

\begin{scope}
  \node at (0,3.0) {(a)};
  \pairbase{black}
  \draw[strand] (-.4,.95)--(-.4,.45) (.4,.95)--(.4,.45);
  \localends
  \node[font=\scriptsize,anchor=east] at (-1.15,.2) {$3$};
  \node[font=\scriptsize,anchor=west] at (1.15,.2) {$-3$};
\end{scope}

\begin{scope}[shift={(4.7,0)}]
  \node at (0,3.0) {(b)};
  \fill[black!20] (-.4,.60) rectangle (.4,.80);
  \pairbase{black}
  \draw[strand] (-.4,.95)--(-.4,.45) (.4,.95)--(.4,.45);
  \draw[densely dashed,line width=.45pt] (-.4,.60)--(.4,.60) (-.4,.80)--(.4,.80);
  \localends
  \node[font=\scriptsize] at (0,1.35) {band};
  \draw[-{Stealth[length=1.2mm]},line width=.45pt] (0,1.15)--(0,.85);
\end{scope}

\draw[-{Stealth[length=1.7mm]},line width=.55pt] (6.55,0)--(7.5,0);
\node[font=\scriptsize] at (7.025,.35) {saddle};

\begin{scope}[shift={(9.4,0)}]
  \node at (0,3.0) {(c)};
  \pairbase{black!55}
  \draw[strand] (-.4,.95) .. controls (-.4,.75) and (.4,.75) .. (.4,.95);
  \draw[strand,black!55] (-.4,.45) .. controls (-.4,.68) and (.4,.68) .. (.4,.45);
  \localends
\end{scope}

\draw[-{Stealth[length=1.7mm]},line width=.55pt] (11.25,0)--(12.2,0);
\node[font=\scriptsize] at (11.725,.35) {isotopy};

\begin{scope}[shift={(14.1,0)}]
  \node at (0,3.0) {(d)};
  \draw[strand] (-1.65,2.2) .. controls (-1.65,.9) and (1.65,.9) .. (1.65,2.2);
  \draw[strand] (-1.65,-2.2) .. controls (-1.65,-.9) and (1.65,-.9) .. (1.65,-2.2);
  \draw[strand,black!55] (0,0) circle[radius=.48];
  \node[font=\small,anchor=west] at (.65,0) {$U$};
  \localends
\end{scope}
\end{tikzpicture}
\caption{The local ribbon move for the adjacent pair $(3,-3)$.
The four exterior ends are fixed throughout.
(a) The original pair. (b) The band attaches to the inner strands
between the first and second crossings. (c) Band surgery replaces the
attaching arcs by the other two sides of the band; the closed component
is drawn in gray. (d) Reidemeister~II cancellations, from top to bottom,
separate the unknot $U$ and leave the two arcs of the shorter pretzel.
The same move works for every odd $q\ge3$.
Compare \cite[Proposition~2.1 and Figure~1]{Lecuona2015}.}
\label{fig:local-ribbon-move}
\end{figure}
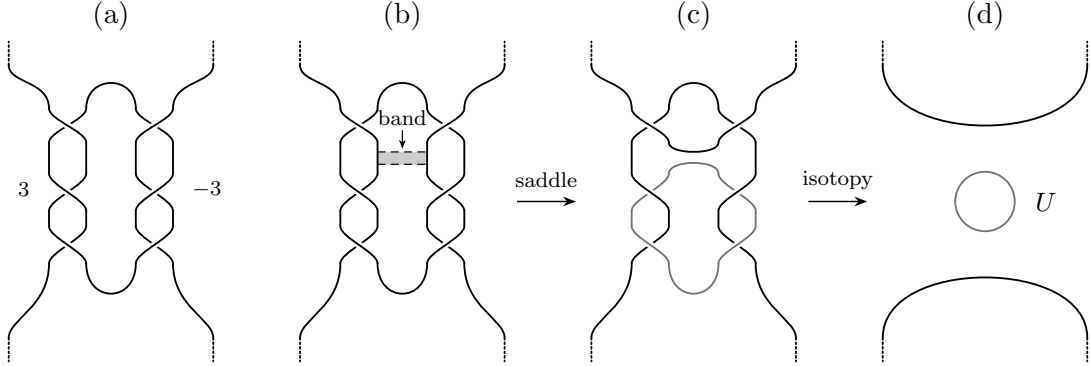

\FloatBarrier

\begin{proposition}\label{prop:pairedreduction}
Let $a\ge3$ be odd and let $q_1,\ldots,q_r$ be positive odd integers.  If $K\in\mathcal P(a;q_1,\ldots,q_r)$, then $\Delta_K(t)$ admits a Fox--Milnor factorization if and only if $\Delta_{P_a}(t)$ does.
\end{proposition}

\begin{proof}
Reordering the parameters of a pretzel knot is achieved by mutation,
and the Alexander polynomial is mutation invariant.  Hence
$\Delta_K(t)\doteq\Delta_{K_0}(t)$, where
\[
K_0=P\!\left(q_1,-q_1,\ldots,q_r,-q_r,a,-a-2,-(a+1)^2/2\right).
\]
For $q_i\ge3$, the ribbon move of
Figures~\ref{fig:global-ribbon-move} and~\ref{fig:local-ribbon-move}
on the adjacent pair $q_i,-q_i$ produces
the split union of an unknot and the pretzel knot obtained by deleting
that pair.  The trace is an orientable pair of pants; capping its split
unknot boundary component gives a smooth annulus, hence a concordance
to the shorter pretzel knot.  When $q_i=1$, the pair $(1,-1)$ can be
deleted by isotopy; see \cite[Section~2]{LecuonaWand2026}.
Successive deletions give a concordance from $K_0$ to $P_a$.
The mutation argument and the ribbon reduction in this setting also
appear in \cite[p.~2169]{Lecuona2015}.
Lemma~\ref{lem:concordanceFM} now gives the equivalence.
\end{proof}

\begin{remark}
There is also a direct algebraic proof of
Proposition~\ref{prop:pairedreduction}.  Belousov's explicit formula
for the Alexander polynomial of a pretzel knot
\cite{Belousov2026} gives, in the present notation,
\[
\Delta_K(t)\doteq
\Delta_{P_a}(t)
\prod_{j=1}^r
\left(\frac{1+t^{q_j}}{1+t}\right)^2 .
\]
The additional factors are reciprocal squares, and hence do not affect
the Fox--Milnor condition.  The proof above records the geometric
concordance obtained by deleting the opposite pairs, whereas
Belousov's formula gives the same reduction directly at the level of
Alexander polynomials.
\end{remark}

\begin{remark}\label{rem:mirrors}
For the mirror $\overline K$ of a knot $K$, one has $\Delta_{\overline K}(t)\doteq\Delta_K(t^{-1})\doteq\Delta_K(t)$.  Thus every Fox--Milnor obstruction obtained here also applies to mirrors.
\end{remark}

By Proposition~\ref{prop:pairedreduction}, it remains to prove the following core statement.

\begin{theorem}\label{thm:core}
Let $a\ge3$ be odd with $a\equiv1\pmod4$.  Then the Alexander polynomial of $P_a$ does not admit a Fox--Milnor factorization.
\end{theorem}

\section{An auxiliary even polynomial}\label{sec:auxiliary}

Write
\[
a=2m-1,\qquad b=2m+1,\qquad k=m-1,\qquad D=ab=4m^2-1.
\]
Throughout the proof of Theorem~\ref{thm:core}, $m\ge3$ is odd and
$k\ge2$ is even.
Lecuona \cite[pp.~2151--2152]{Lecuona2015} computed the following
representative of $\Delta_{P_a}$:
\begin{equation}\label{eq:alex}
\Delta_a(t)=
\frac{(t^{a+2}+1)(t^a+1)}{(t+1)^2}
-m^2t^{a-1}(t-1)^2.
\end{equation}
It is monic of degree $2a$ and satisfies
$\Delta_a(0)=\Delta_a(1)=1$.

The substitution $t=(1+u)/(1-u)$ changes the involution
$t\mapsto t^{-1}$ into $u\mapsto-u$.  Set
\[
E_m(u)=\frac{(1+u)^{2m}+(1-u)^{2m}}2,
\qquad
\mathcal D_m(u)=(1-u)^{2a}
\Delta_a\!\left(\frac{1+u}{1-u}\right).
\]
After substitution and denominator clearing, the first term of
\eqref{eq:alex} becomes
$E_m(u)^2+u^2(1-u^2)^{2m-1}$, and the second becomes
$4m^2u^2(1-u^2)^{2m-2}$.  Hence
\begin{equation}\label{eq:normu}
\mathcal D_m(u)
=E_m(u)^2-u^2(1-u^2)^{2k}(u^2+D).
\end{equation}

To factor this difference of squares, introduce $v$ with
$v^2-u^2=D$.  Taking $s=u+v$ gives the parametrization
\begin{equation}\label{eq:param}
u=\frac12\left(s-\frac Ds\right),
\qquad
v=\frac12\left(s+\frac Ds\right).
\end{equation}
Define
\[
L(s)=E_m(u)+u(1-u^2)^k v.
\]
The substitution $s\mapsto-D/s$ fixes $u$ and negates $v$;
$s\mapsto-s$ negates both.  Thus
\begin{equation}\label{eq:Lnorm}
\mathcal D_m(u(s))=L(s)L(-D/s),
\qquad L(-s)=L(s).
\end{equation}

We next clear the denominators in $L$.  Put
\[
c=D+2,\qquad
X=(s-a)(s+b)=s^2+2s-D,\qquad
Y=(s-b)(s+a)=s^2-2s-D.
\]
Since $X-Y=4s$, the expression
\begin{equation}\label{eq:B0}
B(s)=\frac{X^m-Y^m}{2s}
=2\sum_{j=0}^{k}X^{k-j}Y^j
\end{equation}
is an integer polynomial.  It is even, because $s\mapsto-s$
interchanges $X$ and $Y$.

\begin{proposition}\label{prop:Q}
The polynomial
\begin{equation}\label{eq:Qcanonical}
Q(s)=B(s)^2+(XY)^k(s^2-c)
\end{equation}
is even, monic, and of degree $2a$, and satisfies
\begin{equation}\label{eq:LQ}
L(s)=\frac{Q(s)}{2^a s^{a-1}}.
\end{equation}
Moreover,
\begin{equation}\label{eq:endQ}
Q(a)=2^{2a}a^{2k},\qquad Q(b)=2^{2a}b^{2k},
\end{equation}
and
\begin{equation}\label{eq:Qzero}
Q(0)=-D^{2k}\ne0.
\end{equation}
\end{proposition}

\begin{proof}
The parametrization gives
\[
1+u=\frac{X}{2s},\qquad 1-u=-\frac{Y}{2s},
\qquad 1-u^2=-\frac{XY}{4s^2}.
\]
Since $k$ is even,
\begin{align*}
2^a s^{a-1}E_m(u)&=\frac{X^{2m}+Y^{2m}}{4s^2},\\
2^a s^{a-1}u(1-u^2)^kv
&=\frac{(s^4-D^2)(XY)^k}{2s^2}.
\end{align*}
Using
\[
B(s)^2=\frac{X^{2m}+Y^{2m}-2(XY)^m}{4s^2},
\qquad XY+s^4-D^2=2s^2(s^2-c),
\]
we obtain \eqref{eq:LQ}.  In \eqref{eq:Qcanonical}, the first term
has degree $4k$, while the second is monic of degree $4k+2=2a$.
Both are even integer polynomials.

At $s=a$ and $s=b$, respectively, $X=0$ and $Y=0$.
Equation~\eqref{eq:B0} gives
\[
B(a)=2^aa^k,\qquad B(b)=2^ab^k,
\]
which proves \eqref{eq:endQ}.  At $s=0$, we have
$X=Y=-D$ and $B(0)=2mD^k$.  Therefore
\[
Q(0)=(4m^2-c)D^{2k}=-D^{2k}.\qedhere
\]
\end{proof}

\begin{lemma}\label{lem:positiveinterval}
For every real $s\in[a,b]$, we have $Q(s)>0$.
\end{lemma}

\begin{proof}
On this interval, \eqref{eq:param} gives $v>0$, and $u$ increases
from $-1$ to $1$, since
\[
u'(s)=\frac12\left(1+\frac D{s^2}\right)>0.
\]
Consequently $v=\sqrt{u^2+D}\le2m$ and
$0\le(1-u^2)^k\le1$.  The polynomial $E_m$ is even with nonnegative coefficients,
so
\[
E_m(u)\ge1+\binom{2m}{2}u^2=1+m(2m-1)u^2.
\]
It follows that
\begin{align*}
L(s)&\ge1+m(2m-1)u^2-2m|u|\\
&=(1-m|u|)^2+m(m-1)u^2>0.
\end{align*}
The last inequality holds also at $u=0$, where the expression is $1$.
The denominator in \eqref{eq:LQ} is positive on $[a,b]$, so $Q(s)>0$.
\end{proof}

\section{Factor descent}\label{sec:descent}

\begin{lemma}\label{lem:coprime}
The Laurent polynomials $L(s)$ and $L(-D/s)$ are coprime in
$\Q[s,s^{-1}]$.
\end{lemma}

\begin{proof}
A common zero $s\in\mathbb C^\times$ would satisfy
\[
E_m(u)=0,\qquad u(1-u^2)^k v=0.
\]
The possibilities $u=0$ and $u=\pm1$ are excluded by
$E_m(0)=1$ and $E_m(\pm1)=2^a$.
If $v=0$, then $u^2=-D$, and
\[
E_m(u)=\sum_{j=0}^{m}\binom{2m}{2j}(-D)^j\equiv1\pmod D.
\]
This integer is nonzero because $D>1$.  Thus no common zero exists.
\end{proof}

\begin{lemma}[Factor descent]\label{lem:descent}
If $\Delta_a(t)$ admits a Fox--Milnor factorization, then there is
a monic polynomial $H\in\Z[s]$ of degree $a$ such that
\begin{equation}\label{eq:Hfactor}
Q(s)=-H(s)H(-s).
\end{equation}
\end{lemma}

\begin{proof}
Since $\Delta_a$ is monic of degree $2a$ with
$\Delta_a(0)=\Delta_a(1)=1$, a Fox--Milnor factorization can be
normalized as
\[
\Delta_a(t)=f(t)t^af(t^{-1}),
\]
where $f\in\Z[t]$ is monic of degree $a$.
Indeed, shifting the Laurent factor to have nonzero constant term
and comparing the extreme coefficients fixes its degree and makes
its leading and constant coefficients units.  If the remaining unit
is $\varepsilon t^a$, then $1=\varepsilon f(1)^2$ gives
$\varepsilon=1$.  Replacing $f$ by $-f$ if necessary makes it monic.

Set
\[
G(u)=(1-u)^af\!\left(\frac{1+u}{1-u}\right),
\qquad J(s)=G(u(s)).
\]
Then $G\in\Z[u]$ has degree at most $a$, so $s^aJ(s)\in\Q[s]$.
Equations~\eqref{eq:Lnorm} and the Fox--Milnor factorization give
\begin{equation}\label{eq:twofactorizations}
L(s)L(-D/s)=J(s)J(-s).
\end{equation}
Let $H$ be the monic greatest common divisor of $Q(s)$ and
$s^aJ(s)$ in $\Q[s]$.

We check multiplicities in \eqref{eq:twofactorizations}.
By \eqref{eq:LQ}, $Q$ and $L$ differ by a unit in
$\Q[s,s^{-1}]$.  Lemma~\ref{lem:coprime} therefore implies that
an irreducible factor of $Q$ has the same multiplicity in $Q$ as
in $J(s)J(-s)$.  Its multiplicity in $J(s)$ is thus at most its
multiplicity in $Q$, and is exactly its multiplicity in $H$.
Here the factor $s^a$ does not affect the greatest common divisor,
since $Q(0)\ne0$.
Because $Q$ is even, the same argument shows that its multiplicity
in $J(-s)$ is its multiplicity in $H(-s)$.
Consequently $Q$ is a constant multiple of $H(s)H(-s)$.
Comparing degrees gives $\deg H=a$, and comparing leading
coefficients gives \eqref{eq:Hfactor}, since $a$ is odd.
Finally, $H$ is a monic divisor of the monic integer polynomial $Q$,
so Gauss's lemma gives $H\in\Z[s]$.
\end{proof}

\section{Endpoint values and the contradiction}\label{sec:endpoints}

For a nonzero integer polynomial $F$, let $\operatorname{cont}(F)$
be the positive greatest common divisor of its coefficients.
Gauss's lemma states that
\[
\operatorname{cont}(FG)=\operatorname{cont}(F)\operatorname{cont}(G).
\]
An integral change of variable $y\mapsto-y-n$, with $n\in\Z$,
preserves content, since it and its inverse preserve coefficient
divisibility.  We will use the following consequence.

\Needspace{7\baselineskip}
\begin{lemma}\label{lem:content}
Let $F\in\Z[y]$ be nonzero, $n\in\Z$, and $N$ a positive integer.
If every coefficient of $F(y)F(-y-n)$ is divisible by $N^2$,
then every coefficient of $F(y)$ is divisible by $N$.
\end{lemma}

\begin{proof}
The content of the product is $\operatorname{cont}(F)^2$.
Thus $N^2\mid\operatorname{cont}(F)^2$, which implies
$N\mid\operatorname{cont}(F)$.
\end{proof}

Assume that a Fox--Milnor factorization exists, and choose $H$ as in
Lemma~\ref{lem:descent}.

\begin{lemma}\label{lem:endpointval}
For $d=a,b$, we have
\begin{equation}\label{eq:Hvals}
|H(d)|=|H(-d)|=2^ad^k,
\qquad H(-d)=-H(d).
\end{equation}
\end{lemma}

\begin{proof}
First take $d\in\{a,b\}$ and substitute $s=dy$.  Since $d\mid D$,
\[
X(dy)=d\left(dy^2+2y-\frac Dd\right),\qquad
Y(dy)=d\left(dy^2-2y-\frac Dd\right).
\]
Equation~\eqref{eq:B0} shows that $d^k$ divides every coefficient of
$B(dy)$, and $(X(dy)Y(dy))^k$ is divisible coefficientwise by
$d^{2k}$.  Thus every coefficient of $Q(dy)$ is divisible by
$d^{2k}$.  Applying Lemma~\ref{lem:content} to
\[
Q(dy)=-H(dy)H(-dy)
\]
gives $d^k\mid H(d)$ and $d^k\mid H(-d)$.

For the power of $2$, substitute $s=a+2y$.  Then
\[
X=4y(a+1+y),\qquad Y=4(y-1)(a+y),
\qquad s^2-c=4\bigl(y(a+y)-m\bigr).
\]
Every coefficient of $B(a+2y)$ is divisible by
$2\cdot4^k=2^a$, while every coefficient of
$(XY)^k(s^2-c)$ is divisible by $16^k\cdot4=2^{2a}$.
Hence $Q(a+2y)$ is coefficientwise divisible by $2^{2a}$.
Writing $F(y)=H(a+2y)$, we have
\[
Q(a+2y)=-F(y)F(-y-a).
\]
Lemma~\ref{lem:content} gives $F\in2^a\Z[y]$.
Evaluating at $y=0,1,-a,-a-1$ shows that $2^a$ divides
$H(a),H(b),H(-a),H(-b)$.

Since $d$ is odd, we have proved $2^ad^k\mid H(\pm d)$.
On the other hand, \eqref{eq:endQ} and \eqref{eq:Hfactor} give
\[
H(d)H(-d)=-(2^ad^k)^2.
\]
Dividing both factors by $2^ad^k$ leaves two integers with product
$-1$.  This proves \eqref{eq:Hvals}.
\end{proof}

\begin{proof}[Proof of Theorem~\ref{thm:core}]
Suppose that $\Delta_a$ admits a Fox--Milnor factorization, and
write
\[
H(a)=\varepsilon\,2^aa^k,\qquad
H(b)=\delta\,2^ab^k,
\qquad \varepsilon,\delta\in\{\pm1\}.
\]
Since $a\equiv-1$ and $b\equiv1\pmod m$, we have
$H(b)\equiv H(-a)=-H(a)\pmod m$.
As $k$ is even, reduction of the displayed endpoint values gives
\[
\delta\,2^a\equiv-\varepsilon\,2^a\pmod m.
\]
The integer $m\ge3$ is odd, so $2^a$ is invertible modulo $m$.
Thus $m$ divides $\delta+\varepsilon\in\{-2,0,2\}$, so
$\delta=-\varepsilon$.  Hence $H(a)H(b)<0$, and the intermediate value theorem gives a zero
$s_0\in(a,b)$ of $H$.  Equation~\eqref{eq:Hfactor} then gives
$Q(s_0)=0$, contrary to Lemma~\ref{lem:positiveinterval}.
\end{proof}

\begin{proof}[Proof of Theorem~\ref{thm:main}]
The case $r=0$ is Theorem~\ref{thm:core}.  Proposition~\ref{prop:pairedreduction}
transfers the failure of the Fox--Milnor condition from $P_a$ to every
knot in $\mathcal P(a;q_1,\ldots,q_r)$.  Failure of the Fox--Milnor
condition rules out algebraic sliceness.  Since every topologically
slice knot is algebraically slice, it also rules out topological
sliceness.
\end{proof}

\subsection*{Consequences}  

\emph{Corollary~\ref{cor:allpaired}} follows by combining
Theorem~\ref{thm:main} with \cite[Theorem~4.2]{KLS2022}, which treats
$a\equiv3\pmod4$, and then applying Proposition~\ref{prop:pairedreduction}.

\emph{Corollary~\ref{cor:exceptional}} follows because Lecuona's exceptional
set is contained in the family of Corollary~\ref{cor:allpaired}.
Thus none of its members is algebraically slice.  This proves
\cite[Conjecture~1.3]{Lecuona2015} and makes the nonexceptional
hypothesis in \cite[Theorem~1.1]{Lecuona2015} unnecessary.  The residual
exceptional family of \cite{LecuonaWand2026} contains no ribbon knots.
The same conclusions hold for mirrors by Remark~\ref{rem:mirrors}.

For \emph{Corollary~\ref{cor:3strand}}, Kim--Lee--Song
\cite[Corollaries~1.2 and~1.3]{KLS2022}, together with
\cite[Corollary~1.4]{Lecuona2015} and
\cite[Theorems~1.2 and~1.6]{Miller2017}, establish the stated
classifications outside the residual families $P_a$ and their mirrors,
where $a\equiv1,97\pmod{120}$.  Both congruence classes lie in
$a\equiv1\pmod4$ and are ruled out by Theorem~\ref{thm:main} with
$r=0$ and Remark~\ref{rem:mirrors}.  The converse in the topological
statement uses the fact that ribbon knots are smoothly slice and
Freedman's theorem that knots with trivial Alexander polynomial are
topologically slice \cite{Freedman1982}.

Finally, \emph{Corollary~\ref{cor:LW}} follows from
Corollary~\ref{cor:exceptional}: no ordering of an exceptional
parameter set defines a ribbon knot.  Hence the classification of
prime fibered ribbon pretzel knots up to reordering in
\cite[Theorem~1.1]{LecuonaWand2026} holds without the
nonexceptional hypothesis.

\bibliographystyle{amsalpha}
\bibliography{pretzel_references}

@article{FoxMilnor1966,
  author  = {Fox, Ralph H. and Milnor, John W.},
  title   = {Singularities of $2$-spheres in $4$-space and cobordism of knots},
  journal = {Osaka Journal of Mathematics},
  volume  = {3},
  year    = {1966},
  pages   = {257--267}
}

@article{Freedman1982,
  author  = {Freedman, Michael H.},
  title   = {The topology of four-dimensional manifolds},
  journal = {Journal of Differential Geometry},
  volume  = {17},
  number  = {3},
  year    = {1982},
  pages   = {357--453}
}

@article{GreeneJabuka2011,
  author  = {Greene, Joshua Evan and Jabuka, Stanislav},
  title   = {The slice-ribbon conjecture for $3$-stranded pretzel knots},
  journal = {American Journal of Mathematics},
  volume  = {133},
  number  = {3},
  year    = {2011},
  pages   = {555--580}
}

@article{KLS2022,
  author  = {Kim, Min Hoon and Lee, Changhee and Song, Minju},
  title   = {Non-slice $3$-stranded pretzel knots},
  journal = {Journal of Knot Theory and Its Ramifications},
  volume  = {31},
  number  = {3},
  year    = {2022},
  pages   = {2250018}
}

@article{Lecuona2015,
  author  = {Lecuona, Ana G.},
  title   = {On the slice-ribbon conjecture for pretzel knots},
  journal = {Algebraic \& Geometric Topology},
  volume  = {15},
  number  = {4},
  year    = {2015},
  pages   = {2133--2173}
}

@article{LecuonaWand2026,
  author  = {Lecuona, Ana G. and Wand, Andy},
  title   = {Fibered ribbon pretzels},
  journal = {Bulletin of the London Mathematical Society},
  volume  = {58},
  number  = {1},
  year    = {2026},
  pages   = {e70271}
}

@article{Levine1969,
  author  = {Levine, Jerome},
  title   = {Invariants of knot cobordism},
  journal = {Inventiones Mathematicae},
  volume  = {8},
  year    = {1969},
  pages   = {98--110}
}

@article{Miller2017,
  author  = {Miller, Allison N.},
  title   = {The topological sliceness of $3$-strand pretzel knots},
  journal = {Algebraic \& Geometric Topology},
  volume  = {17},
  number  = {5},
  year    = {2017},
  pages   = {3057--3079}
}

@article{Belousov2026,
  author  = {Belousov, Yury},
  title   = {Explicit Formulae for the Alexander Polynomial of Pretzel Knots},
  journal = {Bull. Aust. Math. Soc.},
  year    = {2026},
  pages   = {1--11},
  doi     = {10.1017/S0004972726101439},
}

\end{document}